\documentclass[a4paper,12pt]{article}
\usepackage[english]{babel}
\usepackage{hyperref}
\usepackage{indentfirst}
\usepackage{amsfonts,amssymb,amsthm,amsmath}
\theoremstyle{definition}
\newtheorem{De}{Definition}
\theoremstyle{theorem}
\newtheorem{Th}{Theorem}
\theoremstyle{lemma}
\newtheorem{Lm}{Lemma}
\theoremstyle{plain}
\newtheorem{Pl}{Proposition}
\theoremstyle{remark}
\newtheorem{Remark}{Remark}
\theoremstyle{plain}
\newtheorem{corollary}{Corollary}
\newcounter{tmp}
\newcounter{tmp1}
\newcommand{\intrd}{\int_{\mathbb{R}^d}}
\newcommand{\intrmb}{\int_{\mathbb{R}^d\setminus B_1}}
\newcommand{\ir}{[0,T]\times \mathbb{R}^{d}}
\newcommand{\ues}{\hat{\mathfrak{u}}_\Delta} 
\begin{document}
\title{Approximate solutions of continuous-time
stochastic games}

\author{Yurii Averboukh\footnote{Krasovskii Institute of Mathematics and Mechanics UrB RAS and Ural Federal University, Yekaterinburg, Russia, e-mail: ayv@imm.uran.ru}}
\maketitle
\begin{abstract}
The paper is concerned with a zero-sum continuous-time stochastic differential game with a dynamics controlled by  a Markov process and a terminal payoff. The value function of the original game is estimated using the value function of a model game. The dynamics of the model game differs from the original one. The general result applied to differential games yields the approximation of value function of differential game by the solution of countable  system of ODEs.

\end{abstract}
\section{Introduction}
Continuous-time dynamical games can be classified as differential games, stochastic differential games and Markov games (or continuous-time stochastic games). For each  type 
the existence theorem for the value function  is proved (see  \cite{bardi}, \cite{elliot_kalton}, \cite{friedman}, \cite{NN_PDG_en}, \cite{varaya_lin} for differential games case, \cite{buslaeva}, \cite{elliot_stochastic}, \cite{Hamadene}, \cite{SDG_book} for stochastic games case and \cite{Guo_Hernandez_lerma}, \cite{Zachrisson} for continuous-time Markov games case). Moreover, it is shown that the value function solves the Isaacs-Bellman equation (see \cite{evans_souganidis}, \cite{Subb_book} for differential games case, \cite{buckdahn_li}, \cite{buslaeva}, \cite{SDG_book} for stochastic games case and \cite{Zachrisson} for continuous-time Markov games case).
The aim of this paper is to provide an approximation of a solution of a continuous-time dynamical game  by a solution of a game  with a different dynamics.

First this problem was considered for particular cases in   \cite{averboukh_dgaa}, \cite{Kol}--\cite{a6}.
In \cite{a4}--\cite{a6} the approximation of the value function of differential game by the value function of stochastic differential game was constructed. In \cite{Kol} (see also \cite{averboukh_dgaa}) the continuous-time Markov game describing the system of interacting particles with the finite number of states is considered by examining the differential game corresponding to the limit case when the number of particles tends to infinity.  It is proved that if the strategy is optimal for the limit game then it is near optimal for the Markov game.

In this paper we consider the following problem: given two stochastic games controlled by Markov processes associated with generators of L\'{e}vy-Khintchine type, construct the strategy in the first game approximating the value function of the second game. To this end we use the  extremal shift first proposed by Krasovskii and Subbotin for differential games \cite{NN_PDG_en} (see also \cite{subb_chen}). In this case the design of strategy relies on a model of the system. 
In the early works by Krasovskii and Subbotin the model was a copy of the original system \cite{NN_PDG_en},  \cite{KrasSubb_guide}. 
Later it was considered  the case when the original system is governed by a delay differential equation and the model is governed by a differential equation  \cite{kras_delay}, \cite{luk_plaks} and the case when the original system is governed by a differential equation whereas the model is described by a stochastic differential equation \cite{a4}--\cite{a6}. In \cite{averboukh_dgaa} the extremal shift is constructed for the case when the original system is the Markov chain describing many particle interacting system and the model is governed by a differential equation.

We construct the   extremal shift  for the first game using the second game as a model. If the player uses this strategy, then her outcome is estimated by the value function of the second game, the rate of the proximity of the original and model systems and the rate of the randomness of the dynamics of both games. Thus, the result is primary applicable for the case when either the original system or the model is deterministic.  We apply it for the case when the first game is a differential game when the second game is a continuous-time Markov game. This yields the approximation of the value function of the differential game by the solution of the  system of countably many ODEs.

The paper is organized as follows. In  Section \ref{sect_def} we describe the examining class of games, define strategies with memory and introduce the assumptions. In Section~\ref{sect_strategies} we define the extremal shift  for a continuous-time stochastic game controlled by a Markov process associated with a generator of L\'{e}vy-Khintchine type and formulate the main result of the paper concerning upper and lower bounds of the value function. In Section~\ref{sect_properties} we prove the main result. In Section \ref{sect_diffgames} we recall the main notions of the theory of differential games. Moreover we  derive the near optimal strategies for the differential game based on solution of the parabolic equation. Note that first this construction was proposed by Krasovskii and Kotelnikova for pursuit-evasion games \cite{a4}--\cite{a6}.  In the last section we present  the approximation of the value of the differential game by the solution of the system of countably many ODEs.

\section{Definitions and assumptions}\label{sect_def}
For $u\in U$ and $v\in V$ let $L^1_t[u,v]:C^2(\mathbb{R}^d)\rightarrow C(\mathbb{R}^d)$ be a generator of L\`{e}vy-Khintchine type  i.e.
\begin{multline*}L^1_t[u,v]\varphi(x)=\frac{1}{2}\langle G^1(t,x,u,v)\nabla,\nabla\rangle\varphi(x)+\langle f^1(t,x,u,v),\nabla\rangle \varphi(x)\\+\int_{\mathbb{R}^d}[\varphi(x+y)-\varphi(x)-\langle y,\nabla \varphi(x)\rangle\mathbf{1}_{B_1}(y)]\nu^1(t,x,u,v,dy)]. \end{multline*} Here $B_1$ denotes the unit ball centered at the origin, $G^1(t,x,u,v)$ is a nonnegative symmetric matrix, $\nu^1(t,x,u,v,\cdot)$ is a measure on $\mathbb{R}^d$ such that
$\nu^1(t,x,u,v,\{0\})=0$ and
$$\intrd \min\{1,y^2\}\nu^1(t,x,u,v,\cdot)dy<\infty. $$

The parameters $u$ and $v$ are considered as  controls of the first and second players respectively. The classes of admissible strategies of the first and second players are described below. Note that under some conditions the operator $L^1_t[u,v]$  generates a stochastic process $X(\cdot)$. 

The first (respectively second) player   wishes to minimize (respectively, maximize) $\mathbb{E}g(X(T))$. In the paper we approximate the value of this game using a solution of a stochastic game with a dynamics governed by a Markov process associated with a generator of L\`{e}vy-Khintchine type
\begin{multline*}L^2_t[u,v]\varphi(x)\triangleq\frac{1}{2}\langle G^2(t,x,u,v)\nabla,\nabla\rangle\varphi(x)+\langle f^2(t,x,u,v),\nabla\rangle \varphi(x)\\+\int_{\mathbb{R}^d}[\varphi(x+y)-\varphi(x)-\langle y,\nabla \varphi(x)\rangle\mathbf{1}_{B_1}(y)]\nu^2(t,x,u,v,dy)]. \end{multline*}

In the general case, $L_t^1[u,v]\neq L_t^2[u,v]$.

To simplify the designations we denote by
$\mathbb{D}_t$ the Skorokhod space $D([t,T],\mathbb{R}^d)$. This set is endowed by the flow of $\sigma$-algebras  $\mathbb{F}_{t,s}\triangleq\mathcal{B}(D([t,s],\mathbb{R}^d))$. Here $\mathcal{B}(S)$ denotes the Borel $\sigma$-algebra on metric space $S$. Recall \cite[Theorem 12.5]{Billingsley} that $$ \mathbb{F}_{t,s}=\sigma\{(\pi_{t_1,\ldots,t_k})^{-1}(A_1,\ldots,A_k):t_1,\ldots,t_k\in [t,s], A_1,\ldots,A_k\subset \mathbb{R}^d\}, $$ where $\pi_{t_1,\ldots,t_k}:\mathbb{D}_t\rightarrow\mathbb{R}^k$ is a projection $\pi_{t_1,\ldots,t_k}x(\cdot)=(x(t_1),\ldots,x(t_k))$.

To simplify the designations put

$$\Sigma^i(t,x,u,v)\triangleq \sum_{j=1}^dG_{jj}^i(t,x,u,v)+\int_{\mathbb{R}^d}\|y\|^2\nu^i(t,x,u,v,dy). $$
$$
b^i(t,x,u,v)\triangleq f^i(t,x,u,v)+\intrmb y\nu^i(t,x,u,v,dy).
$$
Note that the action of the generator $L^i_t[u,v]$ on the function $x\mapsto \langle a,x\rangle$ for any constant $a\in\mathbb{R}^d$ is equal to $\langle a,b^i(t,x,u,v)\rangle.$ Analogously, if $\vartheta_a(x)=\|x-a\|^2$ then
\begin{equation}\label{action_quadratic}
  L^i_t[u,v]\vartheta_a(x)=\Sigma^i(t,x,u,v)+2\langle x-a,b^i(t,x,u,v)\rangle.
\end{equation}

Further, let $\mathcal{A}$ denote the set of functions $\alpha:\mathbb{R}\rightarrow [0,\infty)$ such that $\alpha(\delta)\rightarrow 0$ as $\delta\rightarrow 0$.

We assume that the sets $U$, $V$, the generators $L^1$, $L^2$ and the function $g$ satisfy the following conditions
\begin{list}{(L\arabic{tmp})}{\usecounter{tmp}}
  \item $U$, $V$ are metric compact;
  \item $G^i$, $f^i$, $i=1,2$, are continuous function;
  \item $\nu^i$, $i=1,2$, are weakly continuous;
  \item there exist functions $\alpha_1^i(\cdot)\in\mathcal{A}$, $i=1,2$, such that for any $t\in [0,T]$, $x\in\mathbb{R}^d$, $u\in U$, $v\in V$
  $$\|b^i(t,x,u,v)-b^i(s,x,u,v)\|^2\leq \alpha^i_1(t-s); $$
  \item there exist constants $M_0^i$ and $M_1^i$, $i=1,2$, such that  for any $t\in [0,T]$, $x\in\mathbb{R}^d$, $u\in U$, $v\in V$
  $$|\Sigma^i(t,x,u,v)|\leq M_0^i,\ \  \|b^i(t,x,u,v)\|\leq M_1^i;$$
  \item there exist constants $K^i$, $i=1,2$, such that  for any $t\in [0,T]$, $x,y\in\mathbb{R}^d$, $u\in U$, $v\in V$
  $$\|b^i(t,x,u,v)-b^i(t,y,u,v)\|\leq K^i\|x-y\|; $$
  \item $g$ is Lipschitz continuous with constant $R$;
  \item (Isaacs condition)  either
  \begin{list}{(\arabic{tmp1})}{\usecounter{tmp1}}
  \item
   for any $t\in [0,T]$, $x,\xi\in\mathbb{R}^d$, $u\in U$, $v\in V$
  $$
  \min_{u\in U}\max_{v\in V}\langle\xi,b^1(t,x,u,v)\rangle\\=
\max_{v\in V}\min_{u\in U}\langle\xi,b^1(t,x,u,v)\rangle.
$$
\end{list}
or
  \begin{list}{(\arabic{tmp1})}{\usecounter{tmp1}}\setcounter{tmp1}{1}
  \item
   for any $t\in [0,T]$, $x,\xi\in\mathbb{R}^d$, $u\in U$, $v\in V$
  $$
  \min_{u\in U}\max_{v\in V}\langle\xi,b^2(t,x,u,v)\rangle\\=
\max_{v\in V}\min_{u\in U}\langle\xi,b^2(t,x,u,v)\rangle.$$
\end{list}
\end{list}

Note  (see \cite[Theorems 5.4.2 and 5.5.1]{Kol_markov}) that under imposed conditions for each $t_0$, initial distribution $m_0$, controls $u\in U$, $v\in V$ and $i=1,2$ there exist a filtered probability space and an adapted process $X$ satisfying ${\rm Law}(X(t_0))=m_0$ and for any $\varphi\in\mathcal{D}^i$
$$\varphi(X(t))-\int_{t_0}^t L^i_\tau[u,v]\varphi(X(\tau))d\tau $$
is a martingale. Here the set $\mathcal{D}^i$ is such that  $\mathcal{D}^i\subset C^2(\mathbb{R}^d)$ and $C_b^2(\mathbb{R}^d)\cup\{\vartheta_a\}_{a\in\mathbb{R}^d}\subset \mathcal{D}^i$.

Now we turn to the formalization of the game with the dynamics given by the generator $L_t^1[u,v]$. We assume that the players use randomized strategies with memory. The following definition is inspired by the definition proposed in \cite[p.~79]{GihmanSkorokhod}.
\begin{De}\label{def_u_strategy}Let $t_0$ be an initial time. A strategy of the first player on $[t_0,T]$ is a 5-tuple $\mathfrak{u}=(\Omega^U,\mathcal{F}^U,\{\mathcal{F}^U_s\}_{s\in [t_0,T]},u_{x(\cdot)},P_{x(\cdot)}^U)$ satisfying the following conditions
\begin{enumerate}
  \item $(\Omega^U,\mathcal{F}^U,\{\mathcal{F}^U_s\}_{s\in [t_0,T]})$ is a filtered space;
  \item for each function $x(\cdot)\in \mathbb{D}_{t_0}$ $u_{x(\cdot)}$ is a $\{\mathcal{F}^U_s\}_{s\in [t_0,T]}$-progressive measurable stochastic process with values in $U$, whereas $P_{x(\cdot)}^U$ is a probability on $(\Omega^U,\mathcal{F}^U,\{\mathcal{F}^U_s\}_{s\in [t_0,T]})$.
  \item if $y(s)=x(s)$ for all $s\in [t_0,t]$ then for any $A\in \mathcal{F}^U_t$ $P_{x(\cdot)}^U(A)=P_{y(\cdot)}^U(A)$ and $u_{x(\cdot)}(s)=u_{y(\cdot)}(s)$ $P_{x(\cdot)}^U$-a.s., $s\in [t_0,t]$;
  \item for any $t\in [t_0,T]$ the function $ (x(\cdot),s,\omega)\mapsto u_{x(\cdot)}(t,\omega)$ is measurable with respect to $\mathbb{F}_{t_0,t}\otimes \mathcal{B}([t_0,t])\otimes \mathcal{F}^U_{t_0,t}$.
\end{enumerate}
\end{De}

A strategy $\mathfrak{u}=(\Omega^U,\mathcal{F}^U,\{\mathcal{F}^U_s\}_{s\in [t_0,T]},u_{x(\cdot)},P_{x(\cdot)}^U)$ is called \textit{stepwise} if there exists a partition $\Delta=\{t_l\}_{l=1}^r$ of the interval $[t_0,T]$ such that
equality $x(t_k)=y(t_k)$, $k=0,\ldots, l-1$ implies that $P_{x(\cdot)}(A)=P_{y(\cdot)}(A)$  for any $A\in \mathcal{F}_{t_l-0}^U$ and $u_{x(\cdot)}(s)=u_{y(\cdot)}(s)$ for $s\in [0,t_l)$.

Note that the presented definition of strategy includes feedback strategies, and randomized feedback strategies.

A strategy of the second player is a 5-tuple $\mathfrak{v}=(\Omega^V,\mathcal{F}^V,\{\mathcal{F}^V_s\}_{s\in [t_0,T]},v_{x(\cdot)},P_{x(\cdot)}^V)$ satisfying conditions similar to the conditions of Definition \ref{def_u_strategy} with $v_{x(\cdot)}$ taking values in $V$.

\begin{De}\label{def_motion}
Let $(t_0,x_0)$ be an initial position and let $\mathfrak{u}=(\Omega^U,\mathcal{F}^U,\{\mathcal{F}^U_s\}_{s\in [0,T]},u_{x(\cdot)},P_{x(\cdot)}^U)$ and  $\mathfrak{v}=(\Omega^V,\mathcal{F}^V,\{\mathcal{F}^V_s\}_{s\in [t_0,T]},v_{x(\cdot)},P_{x(\cdot)}^V)$ be strategies of the first and the second players respectively. A 5-tuple $(\Omega^X,\mathcal{F}^X,\{\mathcal{F}^X_s\}_{s\in [t_0,T]},X(\cdot),P)$ is a \textit{realization of the motion generated by the strategies} $\mathfrak{u}$, $\mathfrak{v}$ \textit{and the initial position} $(t_0,x_0)$ if the following conditions hold true.
\begin{enumerate}
  \item $(\Omega^X,\mathcal{F}^X,\{\mathcal{F}^X_s\}_{s\in [t_0,T]})$ is a filtered space.
    \item $P$ is a probability on $(\Omega,\mathcal{F},\{\mathcal{F}_s\}_{s\in [t_0,T]})$, where $\Omega\triangleq\Omega^X\times \Omega^U\times \Omega^V$, $\mathcal{F}\triangleq\mathcal{F}^X\otimes\mathcal{F}^U\otimes\mathcal{F}^V$, $\mathcal{F}_s\triangleq\mathcal{F}^X_s\otimes\mathcal{F}^U_s\otimes\mathcal{F}^V_s$.
  \item $X(\cdot)$ is a $\{\mathcal{F}_s\}_{s\in [t_0,T]}$-adapted process on $(\Omega,\mathcal{F},\{\mathcal{F}_s\}_{s\in [t_0,T]})$ with values in $\mathbb{R}^d$.
  \item $X(t_0)=x_0$ $P$-a.s.
  \item The process
  \begin{equation}\label{dynkin}
  \varphi(X(t))-\int_{t_0}^tL^1_\tau[u(\tau),v(\tau)]\varphi(X(\tau))d\tau
  \end{equation} is a $\{\mathcal{F}_s\}_{s\in [t_0,T]}$-martingale. Here  $u$ and $v$ are stochastic processes defined by the rules
  $$u(\tau,\omega^X,\omega^U,\omega^V)\triangleq u_{X(\cdot,\omega^X,\omega^U,\omega^V)}(\tau,\omega^U),$$ $$ v(\tau,\omega^X,\omega^U,\omega^V)\triangleq v_{X(\cdot,\omega^X,\omega^U,\omega^V)}(\tau,\omega^V),$$
  where $(\omega^X,\omega^U,\omega^V)\in \Omega$;
  \item For any  $x(\cdot)\in \mathbb{D}_{t_0}$ and  any random variable $\zeta'$ on $(\Omega^U,\mathcal{F}^U,\{\mathcal{F}_s^U\}_{s\in [t_0,T]})$
      $$\mathbb{E}^U_{x(\cdot)}\zeta'=\mathbb{E}(\zeta'|X(\cdot)=x(\cdot)),$$ where $\mathbb{E}^U_{x(\cdot)}$ denotes the expectation corresponding to the probability $P^U_{x(\cdot)}$.
  \item  For any $x(\cdot)\in \mathbb{D}_{t_0}$ and any random variable $\zeta''$ on $(\Omega^V,\mathcal{F}^V,\{\mathcal{F}_s^V\}_{s\in [t_0,T]})$
      $$\mathbb{E}^V_{x(\cdot)}\zeta''=\mathbb{E}(\zeta''|X(\cdot)=x(\cdot)),$$ where $\mathbb{E}^V_{x(\cdot)}$ denotes the expectation corresponding to the probability $P^V_{x(\cdot)}$.
  \end{enumerate}
\end{De}
\begin{Remark}
If the strategies $\mathfrak{u}$ and $\mathfrak{v}$ are stepwise then there exists at least one realization of the corresponding motion. To show this consider the set of times of control correction $\{\tau_i\}_{i=1}^n$. The controls of the players are constant on each interval $[\tau_i,\tau_{i+1})$ and they are determined only by values $X(\tau_0),\ldots, X(\tau_i)$.  There exists a realization of the motion on each interval $[\tau_i,\tau_{i+1}]$. Combining this realization  one can construct the realization of the motion on the whole interval $[t_0,T]$.
\end{Remark}

Given the strategies $\mathfrak{u}$, $\mathfrak{v}$, the outcome is not defined in the unique way. The values
\begin{multline*}J^*(t_0,x_0,\mathfrak{u},\mathfrak{v})\triangleq \sup\{\mathbb{E}g(X(T)):
(\Omega^X,\mathcal{F},\{\mathcal{F}^X_s\}_{s\in [0,T]},X(\cdot),P)\mbox{ realizing a }\\ \mbox{motion generated by the strategies } \mathfrak{u}\mbox{ and }\mathfrak{v}\mbox{ and the initial position }(t_0,x_0)\},\end{multline*}
\begin{multline*}J_*(t_0,x_0,\mathfrak{u},\mathfrak{v})\triangleq \inf\{\mathbb{E}g(X(T)):
(\Omega^X,\mathcal{F},\{\mathcal{F}^X_s\}_{s\in [0,T]},X(\cdot),P)\mbox{ realizing a }\\ \mbox{motion generated by the strategies } \mathfrak{u}\mbox{ and }\mathfrak{v}\mbox{ and the initial position }(t_0,x_0)\}\end{multline*}
are the upper and lower outcomes according to the strategies $\mathfrak{u}$ and $\mathfrak{v}$. The upper value of the game is
$${\rm Val}_+(t_0,x_0)=\inf_{\mathfrak{u}}\sup_{\mathfrak{v}}J^*(x_0,\mathfrak{u},\mathfrak{v}).$$ The lower value is equal to
$${\rm Val}_-(t_0,x_0)=\sup_{\mathfrak{v}}\inf_{\mathfrak{u}}J_*(x_0,\mathfrak{u},\mathfrak{v}).$$ Obviously,
$${\rm Val}_-(t_0,x_0)\leq {\rm Val}_+(t_0,x_0). $$ Below we estimate this values using strategies based on the model of the game.

\section{Extremal shift for continuous-time Markov games}\label{sect_strategies}
If $A$ is a metric space then denote by ${\rm rpm}(A)$ the set of Radon probabilities on $A$. If $A$  is a compact then the ${\rm rpm}(A)$  is also a  compact \cite{Warga}.
Below if $\mu$ is a function with values in ${\rm rpm}(A)$, $t\in [0,T]$, $B\subset A$ we write $\mu(t,B)$ instead of $\mu(t)(B)$. If the function $\mu$ takes values in ${\rm rpm}(U)$ (respectively, in ${\rm rpm}(V)$) then it is called generalized control of the first (respectively, second) player.

\begin{De}\label{def_u_stable}
A function $c_+:\ir\rightarrow \mathbb{R}$ is said to be $u$-stable with respect to the generator $L^2$ if
\begin{enumerate}
\item $c_+(T,x)=g(x)$;
\item for any $t,\theta\in [0,T]$, $t<\theta$ there exists a filtered space $(\widetilde{\Omega}^{t,\theta},\widetilde{\mathcal{F}}^{t,\theta}, \{\widetilde{\mathcal{F}}^{t,\theta}_s\}_{s\in [t,\theta]})$ such that for any $\xi\in\mathbb{R}^d$, $v\in Q$ one can find a  $\{\widetilde{\mathcal{F}}_{s}^{t,\theta}\}_{s\in [t,\theta]}$-progressively measurable generalized control of the first player on $[t,\theta]$ $\mu^{t,\theta}_{\xi,v}$, a $\{\widetilde{\mathcal{F}}_{s}^{t,\theta}\}_{s\in [t,\theta]}$-adapted process $Y^{t,\theta}_{\xi,v}$ with values in $\mathbb{R}^d$, and a probability $\widetilde{P}^{t,\theta}_{\xi,v}$ on $\widetilde{\Omega}^{t,\theta}$ such that $Y^{t,\theta}_{\xi,v}(t)=\xi$ $\widetilde{P}^{t,\theta}_{\xi,v}$-a.s., for any $\varphi\in\mathcal{D}^2$
    \begin{equation}\label{L_2_dynkin}
   \varphi(Y^{t,\theta}_{\xi,v}(s))-\int_{t}^s\int_UL^2_\tau[w,v] \varphi(Y^{t,\theta}_{\xi,v}(\tau))\mu^{t,\theta}_{\xi,v}(\tau,dw)d\tau
    \end{equation} is a $\{\widetilde{\mathcal{F}}^{t,\theta}_s\}_{s\in [t,\theta]}$-martingale
    and
   \begin{equation}\label{L_2_u_stable}
    c_+(t,\xi)\geq \widetilde{\mathbb{E}}^{t,\theta}_{\xi,v}c_+(\theta,Y^{t,\theta}_{\xi,v}(\theta));
   \end{equation}
   \item for any random variable  $\phi$ on $\widetilde{\Omega}^{t,\theta}$ dependence of $\widetilde{E}^{t,\theta}_{\xi,v}\phi$ on  $\xi$ and $v$ is measurable;
   \item for any $\varphi\in \mathcal{D}^2$ the function $(\xi,v,s)\mapsto \widetilde{E}^{t,\theta}_{\xi,v}\varphi(Y^{t,\theta}_{\xi,v}(s))$ is measurable.
\end{enumerate}
Here $\widetilde{\mathbb{E}}^{t,\theta}_{\xi,v}$ denotes the expectation corresponding to the probability $\widetilde{P}^{t,\theta}_{\xi,v}$.
\end{De}

The proposed definition of $u$-stability generalizes the definition given by Krasovskii and Subbotion for differential games. This is proved in Proposition~\ref{pl_u_stable_1} below.  Theorem~\ref{th_estimate_first} provides the estimate of the function  ${\rm Val}_+(t_0,x_0)$ by the $u$-stable  function $c_+$.

To estimate the function ${\rm Val}_-(t_0,x_0)$ we will use $v$-stable functions.
\begin{De}\label{def_v_stable}
A function $c_-:\ir\rightarrow \mathbb{R}$ is  $v$-stable with respect to the generator $L^2$ if
\begin{enumerate}
\item $c_-(T,x)=g(x)$;
\item for any $t,\theta\in [0,T]$, $t<\theta$ there exists a filtered space $(\overline{\Omega}^{t,\theta},\overline{\mathcal{F}}^{t,\theta},\{\overline{\mathcal{F}}^{t,\theta}_s\}_{s\in [t,\theta]})$ such that for any $\xi\in\mathbb{R}^d$, $u\in U$ one can find a  $\{\overline{\mathcal{F}}_{s}^{t,\theta}\}_{s\in [t,\theta]}$-progressively measurable generalized control of the second player on $[t,\theta]$ $\mu^{t,\theta}_{\xi,u}$, a $\{\overline{\mathcal{F}}_{s}^{t,\theta}\}_{s\in [t,\theta]}$-adapted process $Y^{t,\theta}_{\xi,u}$ with values in $\mathbb{R}^d$, and a probability $\overline{P}^{t,\theta}_{\xi,u}$ on $\overline{\Omega}^{t,\theta}$ such that $Y^{t,\theta}_{\xi,u}(t)=\xi$ $\overline{P}^{t,\theta}_{\xi,u}$-a.s., for any $\varphi\in\mathcal{D}^2$
   $$
   \varphi(Y^{t,\theta}_{\xi,u}(s))-\int_{t}^s\int_VL^2_\tau[u,w] \varphi(Y^{t,\theta}_{\xi,u}(\tau))\mu^{t,\theta}_{\xi,u}(\tau,dw)d\tau
   $$
    is a $\{\overline{\mathcal{F}}^{t,\theta}_s\}_{s\in [t,\theta]}$-martingale
    and
   $$
    c_-(t,\xi)\leq \overline{\mathbb{E}}^{t,\theta}_{\xi,u}c_-(\theta,Y^{t,\theta}_{\xi,u}(\theta));
   $$
   \item for any random variable  $\phi$ on $\overline{\Omega}^{t,\theta}$ dependence of $\overline{E}^{t,\theta}_{\xi,u}\phi$ on  $\xi$ and $u$ is measurable;
   \item for any $\varphi\in \mathcal{D}^2$ the function $(\xi,u,s)\mapsto \overline{E}^{t,\theta}_{\xi,u}\varphi(Y^{t,\theta}_{\xi,u}(s))$ is measurable.
\end{enumerate}
Here $\overline{\mathbb{E}}^{t,\theta}_{\xi,u}$ denotes the expectation corresponding to the probability $\overline{P}^{t,\theta}_{\xi,u}$.
\end{De}

Given a $u$-stable function $c_+$ and a partition $\Delta=\{t_l\}_{l=0}^r$ of the interval $[t_0,T]$ we  define the stepwise strategy $\ues=(\Omega^U,\mathcal{F}^U,\{\mathcal{F}^U_s\}_{s\in [t_0,T]},u_{x(\cdot)},P_{x(\cdot)})$ by the rules~(\ref{u_opt_space_def}),~(\ref{u_opt_def}) (see below). To this end we need some additional notion.

If condition (L8)--(1) is fulfilled then put
  $$
    \varpi(t,z,\xi,u,v)\triangleq \langle z-\xi,b^1(t,z,u,v)\rangle,
  $$
otherwise, put
  $$
    \varpi(t,z,\xi,u,v)\triangleq \langle z-\xi,b^2(t,\xi,u,v)\rangle.
  $$
For $z,\xi\in\mathbb{R}^d$ let $u_l[z,\xi]$ and $v_l[z,\xi]$ satisfy the condition
\begin{multline*}
  \max_{v\in V}\varpi(t_l,z,\xi,u_l[z,\xi],v)=\min_{u\in U}\max_{v\in V}\varpi(t_l,z,\xi,u,v) \\
  =\max_{v\in V}\min_{u\in U}\varpi(t_l,z,\xi,u,v)=\min_{u\in U}\varpi(t_l,z,\xi,u,v_l[z,\xi]).
\end{multline*}
The functions $(z,\xi)\mapsto u_l[z,\xi]$ and $(z,\xi)\mapsto v_l[z,\xi]$ can be chosen to be measurable.

To define the strategy $\ues$ we construct a sequence of models of the game.
\begin{De}\label{def_model}
A 7-tuple $(\Gamma^l,\mathcal{G}^l,\{\mathcal{G}^l_{s}\}_{s\in [t_0,t_l]},P^l_{x(\cdot)},\hat{\mu}^l_{x(\cdot)},\hat{v}^l_{x(\cdot)},Y_{x(\cdot)}^l)$ is called a model of the game 
for the partition $\Delta$ and the number $l$ if  the following conditions hold true:
\begin{enumerate}
\item for each $x(\cdot)\in\mathbb{D}_{t_0}$ $P_{x(\cdot)}^l$ is a probability on $\Gamma^l$,  $\hat{\mu}_{x(\cdot)}^l$ is a $\{\mathcal{G}^l_s\}_{s\in [t_0,t_l]}$-progressively measurable generalized control of the first player, $\hat{v}_{x(\cdot)}^l$ is a $\{\mathcal{G}^l_s\}_{s\in [t_0,t_l]}$-progressively measurable control of the second player, whereas $Y_{x(\cdot)}^l$ is a c\`{a}dl\`{a}g $\{\mathcal{G}^l_s\}_{s\in [t_0,t_l]}$-adapted process with values in $\mathbb{R}^d$;
\item $Y_{x(\cdot)}^l(\tau)$, $\hat{\mu}^l_{x(\cdot)}(\tau,\cdot)$,  $\hat{v}^l_{x(\cdot)}(\tau)$ depend only on $x(t_0),\ldots,x(t_j)$ where $t_j$ is the greatest element of $\Delta$ such that $t_j\leq \tau$;
\item for any $k=\overline{0,l}$
and $\tau\in [t_k,t_{k+1})$
$$  \max_{v\in Q}\min_{u\in U}\varpi(t_k, x(t_k),Y_{x(\cdot)}^l(t_k),u,v)
  =\min_{u\in U}\varpi(t_k, x(t_k),Y_{x(\cdot)}^l(t_k),u,\hat{v}_{x(\cdot)}^l(\tau));
$$
\item
$ Y^l_{x(\cdot)}(t_0)=x(t_0)\ \ P^l_{x(\cdot)}-\mbox{a.s.},$
and
for any $\varphi\in\mathcal{D}^2$
$$
  \varphi(Y_{x(\cdot)}^l(s))- \int_{t_0}^s\int_U L_\tau^2[w,\hat{v}^l_{x(\cdot)}(\tau)]\varphi(Y^l_{x(\cdot)}(\tau))\hat{\mu}^l_{x(\cdot)}(\tau,dw)d\tau
$$
is a martingale;
\item $\mathbb{E}_{x(\cdot)}^lc_+(t_0,Y^l_{x(\cdot)}(t_0))\geq\ldots\geq \mathbb{E}_{x(\cdot)}^lc_+(t_l,Y_{x(\cdot)}^l(t_l)),$ where $\mathbb{E}^l_{x(\cdot)}$ denotes the expectation corresponding to the probability $P_{x(\cdot)}^l$.
\end{enumerate}
\end{De}

\begin{Pl} Assume that $c_+$ is $u$-stable with respect to generator $t_l$. Additionally, let $\Delta=\{t_l\}_{l=0}^r$ be a   partition of $[t_0,T]$. Then for any $l=1,\ldots,r$ there exists a model for the partition $\Delta$ and the number $l$. 
\end{Pl}
\begin{proof}
We construct the models  inductively.

First, put $\Gamma^1\triangleq \widetilde{\Omega}^{t_0,t_1}$, $\mathcal{G}^1\triangleq \widetilde{\mathcal{F}}^{t_0,t_1}$, $\mathcal{G}^1_s\triangleq \widetilde{\mathcal{F}}^{t_0,t_1}_s$.  For $\tau\in [t_0,t_1]$ set $$ \hat{v}^1_{x(\cdot)}(\tau)\triangleq v_0[x(t_0),x(t_0)].$$   Put $$\hat{\mu}^1_{x(\cdot)}(\tau,\cdot)\triangleq \mu^{t_0,t_1}_{x(t_0),v^1}(\tau,\cdot),\ \  Y_{x(\cdot)}^1(\tau)\triangleq Y^{t_0,t_1}_{x(t_0),v^1}(\tau),\ \ P^1_{x(\cdot)}\triangleq \widetilde{P}^{t_0,t+1}_{x(t_0),v^1}$$ for $v^1=v_0[x(t_0),x(t_0)]$. Obviously, $(\Gamma^1,\mathcal{G}^1,\{\mathcal{G}^l_s\}_{s\in [t_0,t_1]},P^1_{x(\cdot)},\hat{\mu}^1_{x(\cdot)},\hat{v}^1_{x(\cdot)},Y^1_{x(\cdot)})$ is a model at $t^1$.

Now assume that the model is constructed for the number $l$. Define the model for  $l+1$ in the following way. Put $\Gamma^{l+1}\triangleq \Gamma^l\times \widetilde{\Omega}^{t_l,t_{l+1}}$,
$$\mathcal{G}^{l+1}_s\triangleq
    \left\{\begin{array}{cc}
        \mathcal{G}^l_s\otimes \widetilde{\mathcal{F}}^{t_l,t_{l+1}}_{t_l}, & s\in [t_0,t_l] \\
        \mathcal{G}^{l}\otimes \widetilde{\mathcal{F}}^{t_l,t_{l+1}}_{s}, & s\in (t_l,t_{l+1}]
                                      \end{array}
\right.$$
$\mathcal{G}^{l+1}\triangleq \mathcal{G}^{l+1}_{t_{l+1}}$.
Now let $x(\cdot)\in\mathbb{D}_{t_0}$. For $\gamma\in \Gamma^l$, $\omega\in \widetilde{\Omega}^{t_l,t_{l+1}}$ put
$$\hat{v}^{l+1}_{x(\cdot)}(\tau,\gamma,\omega)\triangleq \left\{
\begin{array}{cc}
  \hat{v}^l_{x(\cdot)}(\tau,\gamma), & \tau\in [t_0,t_l), \\
  v_l[x(t_l),Y^l_{x(\cdot)}(t_l,\gamma)], & \tau\in [t_l,t_{l+1}].
\end{array}
\right.
 $$
Choosing $v^{l+1}(\gamma)=v_l[x(t_l),Y^l_{x(\cdot)}(t_l,\gamma)]$, $y^l(\gamma)=Y^l_{x(\cdot)}(t_l,\gamma)$ put
$$
\hat{\mu}^{l+1}_{x(\cdot)}(\tau,\gamma,\omega)\triangleq \left\{
\begin{array}{cc}
  \hat{\mu}^l_{x(\cdot)}(\tau,\gamma), & \tau\in [t_0,t_l), \\
  {\mu}^{t_l,t_{l+1}}_{y^l(\gamma),v^{l+1}(\gamma)}(\tau,\omega), & \tau\in [t_l,t_{l+1}],
\end{array}
\right.
$$
$$
{Y}^{l+1}_{x(\cdot)}(\tau,\gamma,\omega)\triangleq \left\{
\begin{array}{cc}
  {Y}^l_{x(\cdot)}(\tau,\gamma), & \tau\in [t_0,t_l), \\
  Y^{t_l,t_{l+1}}_{y^l(\gamma),v^{l+1}(\gamma)}(\tau,\omega), & \tau\in [t_l,t_{l+1}].
\end{array}
\right.
$$
Finally, let $m^l_{x(\cdot)}$ be a probability on $\mathbb{R}^d$ defined by the rule $m^l_{x(\cdot)}(Z)\triangleq P^l_{x(\cdot)}\{\gamma\in\Gamma^l:Y^l_{x(\cdot)}(t_l,\gamma)\in Z\}$.
If $A\in \Gamma^l$, $B\in \widetilde{\Omega}^{t_l,t_{l+1}}$ then put
$$P^{l+1}_{x(\cdot)}(A\times B)=\intrd P^l_{x(\cdot)}(A|Y_{x(\cdot)}^l(t_l)=z) \widetilde{P}^{t_l,t_{l+1}}_{z,v_l[x(t_l),z]}(B)m^l_{x(\cdot)}(dz). $$
By construction the 7-tuple $(\Gamma^{l+1},\mathcal{G}^{l+1},\{\mathcal{G}^{l+1}_{s}\}_{s\in [t_0,t_l]},P^{l+1}_{x(\cdot)},u_{x(\cdot)}^{l+1},\hat{\mu}^{l+1}_{x(\cdot)},\hat{v}^{l+1}_{x(\cdot)},Y_{x(\cdot)}^{l+1})$ is a model for $l+1$.
\end{proof}

To define the strategy $\ues$ consider $(\Gamma^{r},\mathcal{G}^{r},\{\mathcal{G}^{r}_{s}\}_{s\in [t_0,t_l]},P^{r}_{x(\cdot)},u_{x(\cdot)}^{r}, \hat{\mu}^{r}_{x(\cdot)},\hat{v}^{l+1}_{x(\cdot)},Y_{x(\cdot)}^{r})$ that is the model of the game for the partition $\Delta$ and the number $r$. The strategy $\hat{\mathfrak{u}}=(\Omega^U,\mathcal{F}^U,\{\mathcal{F}^U_s\}_{s\in [t_0,T]},u_{x(\cdot)},P_{x(\cdot)}^U)$ is defined by the rules
\begin{equation}\label{u_opt_space_def}
\Omega^U\triangleq\Gamma^r,\ \ \mathcal{F}^U\triangleq\mathcal{G}^r,\ \ \mathcal{F}_s^U\triangleq\mathcal{G}^r_s, \ \  P^U_{x(\cdot)}\triangleq P^r_{x(\cdot)},
\end{equation}
 for $\tau\in [t_l,t_{l+1})$
\begin{equation}\label{u_opt_def}
u_{x(\cdot)}(\tau,\omega^U)\triangleq u_l[x(t_l),Y^r(t_l,\omega^U)].
\end{equation}

Below we use the following designations:
\begin{equation}\label{kappa_def}\varkappa\triangleq \sup_{t\in [0,T],x\in\mathbb{R}^d,u\in U,v\in V}\|b^1(t,x,u,v)-b^2(t,x,u,v)\|^2\end{equation}
\begin{equation}\label{theta_def}
\Theta=\varkappa+M_0^1+M_0^2,
\end{equation} where constants $M_0^1,M_0^2$ are introduced in condition (L5). Further, set
\begin{equation}\label{beta_def_i}
  \beta\triangleq 2+2K^i.
\end{equation} In formula (\ref{beta_def_i}) $i=1$ if (L8)--(1) is fulfilled and $i=2$ in the opposite case,  $K^i$, $i=1,2$, are Lipschitz constants for functions $x\mapsto b^i(t,x,u,v)$ (see condition (L6)). Additionally, put
\begin{equation}\label{C_def}
C\triangleq\sqrt{Te^{\beta T}}.
\end{equation}

Recall that the payoff function $g$ is Lipschitz continuous with constant~$R$.

\begin{Th}\label{th_estimate_first}
If $c_+$ is $u$-stable with respect to $L^2$ then for any $(t_0,x_0)\in \ir$
$$\lim_{\delta\downarrow 0}\sup\{J(t_0,x_0,\ues,\mathfrak{v}):d(\Delta)\leq \delta\}\leq c_+(t_0,x_0)+ R\cdot C\sqrt{\Theta}.$$
\end{Th}
\begin{corollary}\label{crl_gamma_1}If $c_+$ is $u$-stable with respect to $L^2$ then for any $(t_0,x_0)\in \ir$
$${\rm Val}_+(t_0,x_0)\leq c_+(t_0,x_0)+R\cdot C\sqrt{\Theta}.$$
\end{corollary}

\begin{corollary}\label{crl_gamma_2}If $c_-$ is $v$-stable with respect to $L^2$ then for any $(t_0,x_0)\in \ir$
$$ c_-(t_0,x_0)-R\cdot C\sqrt{\Theta}\leq {\rm Val}_-(t_0,x_0).$$
\end{corollary}

The proof of Theorem \ref{th_estimate_first} is given in the next section.  Corollary \ref{crl_gamma_1} directly follows from Theorem \ref{th_estimate_first}. To prove the Corollary \ref{crl_gamma_2} it suffices to consider the game with payoff function given by $-g$ and interchange the players.

\section{Properties of the model of the game}\label{sect_properties}

Let a 5-tuple $(\Omega^X,\mathcal{F}^X,\{\mathcal{F}^X_s\}_{s\in [t_0,T]},X(\cdot),P)$ be a realization of  the motion for the strategy of the first player $\ues$ and some strategy of the second player $\mathfrak{v}$, partition $\Delta=\{t_l\}_{l=1}^r$ and the initial position $(t_0,x_0)$. Recall (see Section \ref{sect_strategies}) that the construction of the strategy $\ues$ relies on model at time $t_r=T$ (see Definition \ref{def_model}). Further, the elements of $\Omega=\Omega^X\times\Omega^U\times \Omega^V$ are the triples $(\omega^X,\omega^U,\omega^V)$. Recall that
$$u(t,\omega^X,\omega^U,\omega^V)=u_{X(\cdot,\omega^X,\omega^U,\omega^V)}(t,\omega^U),\ \ v(t,\omega^X,\omega^U,\omega^V)=v_{X(\cdot,\omega^X,\omega^U,\omega^V)}(t,\omega^V).$$
Let $(\Gamma^{r},\mathcal{G}^{r},\{\mathcal{G}^{r}_{s}\}_{s\in [t_0,t_l]},P^{r}_{x(\cdot)},u_{x(\cdot)}^{r}, \hat{\mu}^{r}_{x(\cdot)},\hat{v}^{l+1}_{x(\cdot)},Y_{x(\cdot)}^{r})$ be the model of the game  used in the definition of the strategy $\hat{\mathfrak{u}}_\Delta$ (see (\ref{u_opt_space_def}), (\ref{u_opt_def})). For $t\in [0,t^r]$, $A\subset U$ put
\begin{equation}\label{u_hat_def}
\hat{\mu}(t,A,\omega^X,\omega^U,\omega^V)\triangleq\hat{\mu}^r_{X(\cdot,\omega^X,\omega^U,\omega^V)}(t,A,\omega^U),
\end{equation}
\begin{equation}\label{v_hat_def}
\hat{v}(t,\omega^X,\omega^U,\omega^V)\triangleq\hat{v}^r_{X(\cdot,\omega^X,\omega^U,\omega^V)}(t,\omega^U),
\end{equation}
\begin{equation}\label{Y_def}
Y(t,\omega^X,\omega^U,\omega^V)\triangleq Y^r_{X(\cdot,\omega^X,\omega^U,\omega^V)}(t,\omega^U).
\end{equation}

\begin{Lm}\label{lm_stability}We have that
\begin{enumerate}
  \item $Y(t_0)=x_0$ $P$-a.s.;
  \item for any $\varphi\in\mathcal{D}^2$
  $$
    \varphi(Y(s))-\int_{t_0}^s\int_{U} L_\tau^2[w,\hat{v}(\tau)]\varphi(Y(\tau))\hat{\mu}(\tau,dw)d\tau
  $$
  is a martingale;
  \item $\mathbb{E}c_+(t_l,Y(t_l))\geq \mathbb{E}c_+(t_{l+1},Y(t_{l+1})). $
  \item for $\tau\in [t_,t_{l+1})$ \begin{equation*}
  \begin{split}
          \max_{v\in V}\varpi(t_l,x(t_l),Y(t_l),u(\tau),v)&=\max_{v\in V}\min_{u\in U}\varpi(t_l,x(t_l),Y(t_l),u,v)\\&=\min_{u\in U}\varpi(t_l,x(t_l),Y(t_l),u,\hat{v}(\tau)).\end{split}
         \end{equation*}
\end{enumerate}
\end{Lm}
The proof of the Lemma directly follows from  (\ref{u_opt_def})--(\ref{Y_def}), the properties of the model of the game for the number $r$ and the construction the strategy $\ues$.

\begin{Lm}\label{lm_quad_1} There exist a function $\alpha_2(\cdot)\in\mathcal{A}$  such that for $t\geq s$ $$\mathbb{E}\|X(t)-X(s)\|^2 \leq M_0^1(t-s)+\alpha_2(t-s)\cdot(t-s).$$
\end{Lm}
\begin{proof}
Since (\ref{dynkin}) is a martingale, taking into account (\ref{action_quadratic}) we have that
\begin{multline}\label{expect_quad}
  \mathbb{E}\|X(t)-X(s)\|^2=\mathbb{E}\left( \mathbb{E}\left(\|X(t)-X(s)\|^2\Bigl| \mathcal{F}_s\right)\right)\\=
  \mathbb{E}\left(\mathbb{E}\left(\int_{s}^tL^1_\tau[u(\tau),v(\tau)]\|X(\tau)-X(s)\|^2d\tau\Bigl|\mathcal{F}_s\right)\right)\\
  = \mathbb{E}\left(\int_{s}^tL^1_\tau[u(\tau),v(\tau)]\|X(\tau)-X(s)\|^2d\tau\right)\\
  =\mathbb{E}\int_{s}^t\Bigl[\Sigma^1(\tau,X(\tau),u(\tau),v(\tau))+2\langle b(\tau,X(\tau),u(\tau),v(\tau)), X(\tau)-X(s)\rangle.\end{multline}
  Using  condition (L5) we obtain that
  $$\mathbb{E}\|X(t)-X(s)\|^2 \\\leq
  (M_0^1+(M_1^1)^2)(t-s)+\int_{s}^t \mathbb{E}\|X(\tau)-X(s)\|^2 d\tau.
$$
 Gronwall's inequality yields the estimate
$$\mathbb{E}\|X(t)-X(s)\|^2\leq (M_0^1+(M_1^1)^2)e^{(t-s)}(t-s). $$

Put $M'\triangleq (M_0^1+(M_1^1)^2)e^{T}$. Since
$$\mathbb{E}\|X(\tau)-X(s)\|\leq \sqrt{\mathbb{E}\|X(\tau)-X(s)\|^2}, $$ we get from (\ref{expect_quad}) the following estimate
$$\mathbb{E}\|X(t)-X(s)\|^2\leq M_0^1(t-s)+\int_{t}^sM_1^1M'\sqrt{\tau-s}d\tau. $$

Finally, put $$\alpha_2(\delta)\triangleq \frac{2}{3}M_1^1M'\delta^{1/2}.$$
\end{proof}
\begin{Lm}\label{lm_quad_2} There exist a function $\alpha_3(\cdot)\in\mathcal{A}$  such that for $t\geq s$ $$\mathbb{E}\|Y(t)-Y(s)\|^2 \leq M_0^2(t-s)+\alpha_3(t-s)\cdot(t-s).$$
\end{Lm}
The proof of this Lemma  is analogous to the proof of the previous Lemma and relies on Lemma \ref{lm_stability} and conditions (L1)--(L7).

\begin{Lm}\label{lm_kras_subb} There exists a function $\epsilon(\cdot)\in\mathcal{A}$ such that
\begin{multline}\label{X_t_lp_Y_t_lp_estima}\mathbb{E}\|X(t_{l+1})-Y(t_{l+1})\|^2\leq \|X(t_{l})-Y(t_l)\|^2(1+\beta(t_{l+1}-t_l))\\+\Theta(t_{l+1}-t_l)+\epsilon(t_{l+1}-t_l)\cdot(t_{l+1}-t_l). \end{multline}
\end{Lm}
\begin{proof}
We have that
\begin{multline*}\|X(t_{l+1})-Y(t_{l+1})\|^2=\|(X(t_{l+1})-X(t_l))-(Y(t_{l+1})-Y(t_l))+(X(t_l)-Y(t_l))\|^2\\=
\|X(t_{l+1})-X(t_l)\|^2+\|X(t_l)-Y(t_l)\|^2+\|Y(t_{l+1})-Y(t_l\|^2\\-2\langle X(t_{l+1})-X(t_l),Y(t_{l+1})-Y(t_l)\rangle+2\langle  X(t_{l+1})-X(t_l),X(t_l)-Y(t_l)\rangle\\-2\langle  Y(t_{l+1})-Y(t_l),X(t_l)-Y(t_l)\rangle\\ \leq \|X(t_l)-Y(t_l)\|^2+
2\|X(t_{l+1})-X(t_l)\|^2+2\|Y(t_{l+1})-Y(t_l)\|^2\\+2\langle  X(t_{l+1})-X(t_l),X(t_l)-Y(t_l)\rangle-2\langle  Y(t_{l+1})-Y(t_l),X(t_l)-Y(t_l)\rangle.
\end{multline*}
Thus, by Lemmas \ref{lm_quad_1} and \ref{lm_quad_2}
\begin{multline}\label{expect_X_Y_1}
\mathbb{E}\|X(t_{l+1})-Y(t_{l+1})\|^2\\\leq 2\mathbb{E}\|X(t_{l+1})-X(t_l)\|^2+2\mathbb{E}\|Y(t_{l+1})-Y(t_l)\|^2+ \mathbb{E}\|X(t_l)-Y(t_l)\|^2\\+2\mathbb{E}\langle  X(t_{l+1})-X(t_l),X(t_l)-Y(t_l)\rangle-2\mathbb{E}\langle  Y(t_{l+1})-Y(t_l),X(t_l)-Y(t_l)\rangle\\\leq \mathbb{E}\|X(t_l)-Y(t_l)\|^2
+2 (M_0^1+M_0^2+\alpha_2(t_{l+1}-t_l)+\alpha_3(t_{l+1}-t_l))(t_{l+1}-t_l)\\+2\mathbb{E}\langle  X(t_{l+1})-X(t_l),X(t_l)-Y(t_l)\rangle-2\mathbb{E}\langle  Y(t_{l+1})-Y(t_l),X(t_l)-Y(t_l)\rangle.
\end{multline}
Further, let us estimate $\mathbb{E}\langle  X(t_{l+1})-X(t_l),X(t_l)-Y(t_l)\rangle$ and $\mathbb{E}\langle  Y(t_{l+1})-Y(t_l),X(t_l)-Y(t_l)\rangle$.
Since (\ref{dynkin}) is a martingale, using formula $L^1_t[u,v]\langle a,x\rangle=\langle a, b^1(t,x,u,v)\rangle$ we obtain that
\begin{multline}\label{expect_X_estima}
  \mathbb{E}\langle  X(t_{l+1})-X(t_l),X(t_l)-Y(t_l)\rangle=\mathbb{E}(\mathbb{E}(\langle  X(t_{l+1})-X(t_l),X(t_l)-Y(t_l)\rangle|\mathcal{F}_{t_l})) \\
  =\mathbb{E}\left(\mathbb{E}\left(\int_{t_l}^{t_{l+1}}L^1_\tau[u(\tau),v(\tau)]\left\langle X(t_l)-Y(t_l),X(\tau)-X(t_l) \right\rangle d\tau\Bigl|\mathcal{F}_{t_l}\right)\right)\\=
  \mathbb{E}\int_{t_l}^{t_{l+1}}L^1_\tau[u(\tau),v(\tau)]\left\langle X(t_l)-Y(t_l),X(\tau)-X(t_l) \right\rangle d\tau\\=
  \mathbb{E}\int_{t_l}^{t_{l+1}}\langle X(t_l)-Y(t_l),b^1(\tau,X(\tau),u(\tau),v(\tau))\rangle.
\end{multline}
It follows from  conditions (L4) and (L6) that for $\tau\in [t_l,t_{l+1}]$
\begin{multline*}
  \left\langle X(t_l)-Y(t_l),b^1(\tau,X(\tau),u(\tau),v(\tau))\right\rangle \\ \leq
  \left\langle X(t_l)-Y(t_l),b^1(t_l,X(t_l),u(\tau),v(\tau))\right\rangle\\+
  \frac{1}{2}\|X(t_l)-Y(t_l)\|^2+K^1\|X(\tau)-X(t_l)\|^2+\alpha_1^1(\tau-t_l).
\end{multline*}
Thus, (\ref{expect_X_estima}) and Lemma \ref{lm_quad_1} yield the inequality
\begin{multline}\label{expect_X_estima_final}
  \mathbb{E}\langle  X(t_{l+1})-X(t_l),X(t_l)-Y(t_l)\rangle \\
  \leq \mathbb{E}\int_{t_l}^{t_{l+1}}\langle X(t_l)-Y(t_l),b^1(t_l,X(t_l),u(\tau),v(\tau))\rangle d\tau\\+\frac{1}{2}\mathbb{E}\|X(t_l)-Y(t_l)\|^2(t_{l+1}-t_l)+
 \alpha_4(t_{l+1}-t_l) (t_{l+1}-t_l).
\end{multline} Here we denote
$\alpha_4(\delta)\triangleq K^1[M_0^1\delta+\alpha_2(\delta)\cdot\delta]
+\alpha_1^1(\delta). $
Note that $\alpha_4(\cdot)\in\mathcal{A}$.

Analogously, using Lemma \ref{lm_quad_2} we obtain that
\begin{multline}\label{expect_Y_estima_final}
  -\mathbb{E}\langle  Y(t_{l+1})-Y(t_l),X(t_l)-Y(t_l)\rangle
  \\\leq -\mathbb{E}\int_{t_l}^{t_{l+1}}\int_U\langle X(t_l)-Y(t_l),b^2(t_l,Y(t_l),w,\hat{v}(\tau))\rangle\hat{\mu}(\tau,dw)d\tau\\ +\frac{1}{2}\mathbb{E}\|X(t_l)-Y(t_l)\|^2(t_{l+1}-t_l)+
 \alpha_5(t_{l+1}-t_l)] (t_{l+1}-t_l).
\end{multline} Here $\alpha_5(\cdot)$ is a function from the set $\mathcal{A}$ given by the rule
$$\alpha_5(\delta)=K^2[M_0^2\delta+\alpha_3(\delta)\cdot\delta]
+\alpha_1^2(\delta). $$
Combining (\ref{expect_X_Y_1}), (\ref{expect_X_estima_final}), (\ref{expect_Y_estima_final}) and Lemmas \ref{lm_quad_1}, \ref{lm_quad_2} we obtain that
\begin{multline}\label{expect_X_Y_2}
\mathbb{E}\|X(t_{l+1})-Y(t_{l+1})\|^2\leq
\mathbb{E}\|X(t_{l})-Y(t_{l})\|^2(1+(t_{l+1}-t_l))\\+(M_0^1+M_0^2+\epsilon(t_{l+1}-t_l))(t_{l+1}-t_l) \\+ \mathbb{E}\int_{t_l}^{t_{l+1}}\int_U[\langle X(t_l)-Y(t_l),b^1(t_l,X(t_l),u(\tau),v(\tau))\rangle\\
-\langle X(t_l)-Y(t_l),b^2(t_l,Y(t_l),w,\hat{v}(\tau))\rangle]\hat{\mu}(\tau,dw)d\tau.
\end{multline}
Here $$\epsilon(\delta)\triangleq \alpha_2(\delta)+\alpha_3(\delta)+\alpha_4(\delta)+\alpha_5(\delta). $$
Now assume that condition (L8)--(1) is fulfilled. Taking into account condition (L6) and definition of $\varkappa$ (see (\ref{kappa_def})) we obtain that for all $w\in U$
\begin{multline}\label{expect_diff_1}
  \langle X(t_l)-Y(t_l),b^1(t_l,X(t_l),u(\tau),v(\tau))-b^2(t_l,Y(t_l),w,\hat{v}(\tau))\rangle\\
  \leq \langle X(t_l)-Y(t_l),b^1(t_l,X(t_l),u(\tau),v(\tau))-b^1(t_l,X(t_l),w,\hat{v}(\tau))\rangle]\\+
(1/2+K^1)\|X(t_l)-Y(t_l)\|^2+\varkappa/2\\
=\varpi(t_l,X(t_l),Y(t_l),u(\tau),v(\tau))-\varpi(t_l,X(t_l),Y(t_l),w,\hat{v}(\tau)\\ +(1/2+K^1)\|X(t_l)-Y(t_l)\|^2 +\varkappa/2.
\end{multline}
If condition (L8)--(2) holds true the inequality (\ref{expect_diff_1}) takes the form
\begin{multline}\label{expect_diff_2}
  \langle X(t_l)-Y(t_l),b^1(t_l,X(t_l),u(\tau),v(\tau))-b^2(t_l,Y(t_l),w,\hat{v}(\tau))\rangle \\
  \leq \varpi(t_l,X(t_l),Y(t_l),u(\tau),v(\tau))-\varpi(t_l,X(t_l),Y(t_l),w,\hat{v}(\tau))\\ +(1/2+K^2)\|X(t_l)-Y(t_l)\|^2 +\varkappa/2.
\end{multline}

The statement 4 of Lemma \ref{lm_stability} yields that for any $\tau\in [t_l,t_{l+1})$, and any $w\in U$
$$
\varpi(t_l,X(t_l),Y(t_l),u(\tau),v(\tau))-\varpi(t_l,X(t_l),Y(t_l),w,\hat{v}(\tau))\leq 0.
$$
This, (\ref{theta_def}), (\ref{expect_X_Y_2}), the definition of $\beta$ (see (\ref{beta_def_i})) and inequalities (\ref{expect_diff_1}), (\ref{expect_diff_2}) imply  inequality~(\ref{X_t_lp_Y_t_lp_estima}).
\end{proof}

\begin{proof}[Proof of Theorem \ref{th_estimate_first}]
By Lemma \ref{lm_kras_subb} we have
$$\mathbb{E}\|X(t_{l+1})-Y(t_{l+1})\|^2\leq e^{\beta(t_{l+1}-t_l)}\mathbb{E}\|X(t_{l})-Y(t_{l})\|^2+[\Theta+\epsilon(d(\Delta))](t_{l+1}-t_l). $$
Therefore,
$$\mathbb{E}\|X(t_{r})-Y(t_{r})\|^2\leq e^{\beta T} \mathbb{E}\|X(t_0)-Y(t_{0})\|^2+e^{\beta T}[\Theta+\epsilon(d(\Delta))]T. $$
Taking into account statement 1 of Lemma \ref{lm_stability} we obtain that
$$\mathbb{E}\|X(t_{r})-Y(t_{r})\|^2\leq C^2[\Theta+\epsilon(d(\Delta))] $$ where the constant $C$ is defined by (\ref{C_def}).

 Jensen's inequality yields the estimate
 \begin{equation}\label{X_Y_abs_estima}
   \mathbb{E}\|X(t_{r})-Y(t_{r})\|\leq C\sqrt{[\Theta+\epsilon(d(\Delta))]}.
 \end{equation}
 We have that
 $$g(X(t_r))=g(Y(t_r))+(g(X(t_r))-g(Y(t_r)))\leq g(Y(t_r))+R\|X(t_r)-Y(t_r)\|. $$
 Further, taking into account (\ref{X_Y_abs_estima}) we get the inequality
 \begin{multline*}J(t_0,x_0,\ues,\mathfrak{v})=\mathbb{E}g(X(t_r)\leq \mathbb{E}g(Y(t_r))+R\mathbb{E}\|X(t_r)-Y(t_r)\|\\\leq \mathbb{E}g(Y(t_r))+RC\sqrt{[\Theta+\epsilon(d(\Delta))]}. \end{multline*}
Statement 3 of Lemma \ref{lm_stability} yields the inequality
$$\mathbb{E}g(Y(t_r))=\mathbb{E}c_+(t_r,Y(t_r))\leq\mathbb{E}c_+(t_0,Y(t_0)). $$
Since $Y(t_0)=X(t_0)$ $P$-a.s., we obtain that
$$J(t_0,x_0,\ues,\mathfrak{v})\leq c_+(t_0,Y(t_0))+RC\sqrt{[\Theta+\epsilon(d(\Delta))]}. $$ Since $\epsilon(\delta)\rightarrow 0$ as $\delta\rightarrow 0$, we get the conclusion of the Theorem.
\end{proof}
\section{Value function of differential game}\label{sect_diffgames}

In this section we consider differential game with the dynamics given by
\begin{equation}\label{system}
  \frac{d}{dt}x(t)=f^1(t,x(t),u(t),v(t)), \ \ t\in [0,T],\ \  x\in\mathbb{R}^d, \ \ u(t)\in U, \ \ v(t)\in V.
\end{equation} This equation corresponds to the generator
  \begin{equation}\label{L1_diff}
  L^1_t[u,v]\varphi(x)=\langle f^1(t,x,u,v),\nabla \varphi(x)\rangle.
\end{equation}
As above the variable $u$ (respectively, $v$) stands for the control of the first (respectively, second) player. The aim of first (respectively, second) player is to minimize (respectively, maximize) the payoff function $g(x(T))$.

Let $$\mathcal{U}[t_0]\triangleq\{u:[t_0,T]\rightarrow U\ \ \mbox{measurable}\},\ \ \mathcal{V}[t_0]\triangleq\{v:[t_0,T]\rightarrow V\ \ \mbox{measurable}\}. $$ The set $\mathcal{U}[t_0]$ (respectively, $\mathcal{V}[t_0]$) is a set of open-loop strategies of the first (respectively, second) player.

We assume that the function $f^1$ is continuous, bounded by $M^1_1$, Lipschitz continuous with respect to $x$ with the constant $K^1$. Additionally, we suppose that the Isaacs condition is fulfilled, i.e. for any $t\in [0,T]$, $x,\xi\in\mathbb{R}^d$
\begin{equation}\label{isaacs_diff}
\min_{u\in U}\max_{v\in V}\langle \xi,f^1(t,x,u,v)\rangle=\max_{v\in V}\min_{u\in U}\langle \xi,f^1(t,x,u,v)\rangle.
\end{equation}


We use the feedback formalization of differential games proposed by Krasovskii and Subbotin. Let $p:\ir\rightarrow U$ be a function, $(t_0,x_0)$ be an initial position, and let $\Delta=\{t_l\}_{l=1}^r$ be a partition of the interval $[t_0,T]$. We say that the strategy $\mathfrak{u}=(\Omega^U,\mathcal{F}^U,\{\mathcal{F}^U_s\}_{s\in [t_0,T]},u_{x(\cdot)},P_{x(\cdot)}^U)$ belongs to the set $\mathbb{U}_{t_0,x_0,\Delta}[p]$ if for any $x(\cdot)\in\mathbb{D}_{t_0}$ satisfying $x(t_0)=x_0$ and $\omega^U\in \Omega^U$, $\tau\in [t_{l},t_{l+1})$ $$u_{x(\cdot)}(\tau,\omega^U)=p(t_l,x(t_l)).$$ Note that the elements of the set $\mathbb{U}_{t_0,x_0,\Delta}$ are stepwise deterministic strategies. Additionally, if $(t_0,x_0)\in \ir$, $\mathfrak{u}\in\mathbb{U}_{t_0,x_0,\Delta}[p]$, $v\in\mathcal{V}$ then  the outcome
$$J(t_0,x_0,\mathfrak{u},v)=J^*(t_0,x_0,\mathfrak{u},v)=J_*(t_0,x_0,\mathfrak{u},v) $$ is well-defined.

Analogously, we say that the strategy $\mathfrak{v}=(\Omega^V,\mathcal{F}^V,\{\mathcal{F}^V_s\}_{s\in [t_0,T]},v_{x(\cdot)},V_{x(\cdot)}^U)$ is an element of the set $\mathbb{V}_{t_0,x_0,\Delta}[q]$ if for any $x(\cdot)\in\mathbb{D}_{t_0}$, such that $x(t_0)=x_0$ and $\omega^V\in \Omega^V$, $\tau\in [t_{l},t_{l+1})$ $$v_{x(\cdot)}(\tau,\omega^V)=q(t_l,x(t_l)).$$ As above, for any $(t_0,x_0)\in \ir$, $\mathfrak{v}\in\mathbb{V}_{t_0,\Delta}[q]$, $u\in\mathcal{U}$   the outcome
$$J(t_0,x_0,u,\mathfrak{v})=J^*(t_0,x_0,u,\mathfrak{v})=J_*(t_0,x_0,u,\mathfrak{v}) $$ is well-defined.

Krasovskii and Subbotin proved that there exist  functions $p^*:\ir \rightarrow U$, $q^*:\ir\rightarrow V$ such that
\begin{multline*}
  \lim_{\delta\downarrow 0}\sup\{J(t_0,x_0,\mathfrak{u},v):\mathfrak{u}\in \mathbb{U}_{t_0,x_0,\Delta}[p^*],d(\Delta)\leq\delta,v\in\mathcal{V}\} \\
  =\lim_{\delta\downarrow 0}\sup\{J(t_0,x_0,\mathfrak{u},v):\mathfrak{u}\in \mathbb{U}_{t_0,x_0,\Delta}[p],d(\Delta)\leq\delta,v\in\mathcal{V},p\in U^{\ir}\}\\=  \lim_{\delta\downarrow 0}\sup\{J(t_0,x_0,u,\mathfrak{v}):\mathfrak{v}\in \mathbb{V}_{t_0,x_0,\Delta}[q^*],d(\Delta)\leq\delta,u\in\mathcal{U}\} \\
  =\lim_{\delta\downarrow 0}\sup\{J(t_0,x_0,u,\mathfrak{v}):\mathfrak{v}\in \mathbb{V}_{t_0,x_0,\Delta}[q],d(\Delta)\leq\delta,u\in\mathcal{U},q\in V^{\ir}\}\\={\rm Val}(t_0,x_0).
\end{multline*} Here $B^A$ stands for the set of functions from $A$ to $B$.

Note that the value function ${\rm Val}$ can be defined using nonanticipating strategies~\cite{bardi}. This formalization is equivalent to Krasovskii--Subbotin approach \cite{Subb_book}.

The function $c_+:\ir\rightarrow \mathbb{R}^d$ is Krasovskii--Subbotin $u$-stable (see \cite{NN_PDG_en}) if $c_+(T,x)=g(x)$ and for any $t,\theta\in [0,T]$, $t<\theta$, $\xi\in\mathbb{R}^d$, $v\in V$ there exists a weakly measurable function $\tau\mapsto \mu_{\xi,v}^{t,\theta}(\tau)\in {\rm rpm}(U)$ such that for $y_{\xi,v}^{t,\theta}(\cdot)$ satisfying
\begin{equation}\label{kras_subb_stab_motion}
\frac{d}{d\tau}y_{\xi,v}^{t,\theta}(\tau)= \int_{U}f^1(\tau,y_{\xi,v}^{t,\theta}(\tau),w,v)\mu_{\xi,v}^{t,\theta}(\tau,dw), \ \ y_{\xi,v}^{t,\theta}(t)=\xi
\end{equation}
 the following inequality holds true \begin{equation}\label{kras_subb_ineq_u}
c_+(t,\xi)\geq c_+(\theta,y_{\xi,v}^{t,\theta}(\theta)).
\end{equation}

Recall \cite{NN_PDG_en} that if $c_+$ is Krasovskii--Subbotin $u$-stable then there exists a function $p:\ir\rightarrow U$ such that
$$\lim_{\delta\downarrow 0}\sup\{J(t_0,x_0,\mathfrak{u},v):\mathfrak{u}\in \mathbb{U}_{t_0,x_0,\Delta}[p],d(\Delta)\leq\delta,v\in\mathcal{V}\}\leq c_+(t_0,x_0).  $$
In addition, \cite[Theorem 6.4]{Subb_book} states  that $c_+$ is Krasovskii--Subbotin $u$-stable function if and only if $c_+$ is a minimax (viscosity) supersolution of the Hamilton-Jacobi PDE
$$\frac{\partial c}{\partial t}+\min_{u\in U}\max_{v\in V}\langle \nabla c,f^1(t,x,u,v)\rangle=0, \ \ c(T,x)=g(x). $$

The link between Krasovskii--Subbotin $u$-stability and the notion of  $u$-stability with respect to the generator introduced in Definition \ref{def_u_stable} is given in the following.
\begin{Pl}\label{pl_u_stable_1}If
$c_+$ is Krasovskii--Subbotin $u$-stable then $c_+$ is $u$-stable with respect to the generator $L^2=\langle f^1(t,x,u,v),\nabla\varphi(x)\rangle$.
\end{Pl}
\begin{proof} Let $t,\theta\in [0,T]$. Put $\widetilde{\Omega}^{t,\theta}\triangleq D([t,\theta],U)$. Let $\widetilde{\mathcal{F}}_s^{t,\theta}\triangleq\mathcal{B}(D([t,s],U))$ be a filtration on $\widetilde{\Omega}^{t,\theta}$, and let $\widetilde{\mathcal{F}}^{t,\theta}\triangleq\widetilde{\mathcal{F}}_\theta^{t,\theta}$. Put $Y^{t,\theta}_{\xi,v}\triangleq y_{\xi,v}^{t,\theta}$. Note that $Y^{t,\theta}_{\xi,v}$ is a deterministic process.
 Finally, let $\widetilde{P}_{\xi,v}^{t,\theta}$ be an arbitrary probability on $\widetilde{\Omega}^{t,\theta}$.
Formula (\ref{kras_subb_stab_motion}) yields  that the process (\ref{L_2_dynkin}) for $L^2_\tau[u,v]\varphi(x)=\langle f^1(t,x,u,v),\nabla\varphi(x)\rangle $ is a martingale, and the equlity $Y_{\xi,v}^{t,\theta}(t)=\xi$. Finally, inequality~(\ref{kras_subb_ineq_u}) implies~(\ref{L_2_u_stable}).
\end{proof}
The notion of Krasovskii--Subbotin $v$-stability is defined in the same way as  $u$-stability. For $v$-stable functions an analog of Proposition \ref{pl_u_stable_1} is also fulfilled.

Now we consider the case when the model system is given by a stochastic differential equation.
\begin{Pl}\label{pl_psi_stable}
If $\psi_\sigma$ is a solution of  $$
  \frac{\partial \psi}{\partial t}+\min_{u\in U}\max_{v\in V}\langle \nabla \psi, f^1(t,x,u,v)\rangle+\frac{\sigma^2}{2}\triangle \psi=0, \ \ \psi(T,x)=g(x).
$$
then $\psi_\sigma$ is $u$- and $v$-stable with respect to $$L^2_\tau[u,v]\varphi(x)=\langle \nabla\varphi(x), f^1(t,x,u,v)\rangle +\frac{\sigma^2}{2}\cdot\triangle\varphi(x).$$
\end{Pl}
\begin{proof}
Put $\widetilde{\Omega}^{t,\theta}\triangleq C([t,\theta],\mathbb{R}^d)$, $\widetilde{\mathcal{F}}_s^{t,\theta}=\mathcal{B}(C([t,\theta],\mathbb{R}^d))$, $\widetilde{\mathcal{F}}^{t,\theta}\triangleq\widetilde{\mathcal{F}}^{t,\theta}_\theta$. Let $\widetilde{P}^{t,\theta}_{\xi,v}$ be a Wiener measure on $\widetilde{\Omega}^{t,\theta}$. 

Further, for the constant second player's control $v\in V$ consider the control problem for the stochastic differential equation
\begin{equation}\label{sde}
dY(\tau)=f^1(\tau,Y(\tau),u(\tau),v)d\tau+\sigma dW(\tau),\ \ Y(t)=\xi
\end{equation} on time interval $[t,\theta]$ with the payoff functional given by $\widetilde{\mathbb{E}}^{t,\theta}_{\xi,v}\psi_\sigma(\theta,Y(\theta))$.

By \cite{Hamadene_1995} there exist a control $u^{t,\theta}_{\xi,v}$ and a function $\rho:\ir \rightarrow \mathbb{R}$ such that
\begin{equation*}\label{HJ_2_order_aux}
  \frac{\partial \rho}{\partial t}+\min_{u\in U}\langle \nabla \rho, f^1(t,x,u,v)\rangle+\frac{\sigma^2}{2}\triangle \rho=0, \ \ \rho(\theta,x)=\psi_\sigma(\theta,x).
\end{equation*}
and for $Y^{t,\theta}_{\xi,v}(\cdot)$  satisfying (\ref{sde}) with $u=u^{t,\theta}_{\xi,v}$
 the inequality
 $$\rho(t,\xi)\geq \widetilde{\mathbb{E}}^{t,\theta}_{\xi,v}\rho(\theta,Y^{t,\theta}_{\xi,v}(\theta)). $$

Using the comparison principle for parabolic equations (see, for example \cite[Theorem I.16]{friedman_parabolic}) we obtain that $$\psi_\sigma(t,\xi)\geq \rho(t,\xi)\geq \widetilde{\mathbb{E}}^{t,\theta}_{\xi}\rho(\theta,Y(\theta))= \widetilde{\mathbb{E}}^{t,\theta}_{\xi,v}\psi_\sigma(\theta,Y^{t,\theta}_{\xi,v}(\theta)).$$ To prove the $u$-stability of the function $\psi_\sigma$  with respect to $L^2_t$ it suffices to put $$\hat{\mu}^{t,\theta}_{\xi,v}\triangleq \delta_{u^{t,\theta}_{\xi,v}}$$ where $\delta_z$ denotes the Dirac measure concentrated at $z$.

The $v$-stability of $\psi_\sigma$ is proved in the same way.
\end{proof}

Theorem \ref{th_estimate_first}, Corollaries \ref{crl_gamma_1}, \ref{crl_gamma_2} and Propositions \ref{pl_u_stable_1}, \ref{pl_psi_stable} imply the following for $C_1\triangleq C\sqrt{d}$.

\begin{corollary}
  Let $\hat{\mathfrak{u}}_\Delta$ be a stepwise strategy constructed by (\ref{u_opt_space_def}) and (\ref{u_opt_def}) for $c_+=\psi_\sigma$. Then
  $$ \lim_{\delta\downarrow 0}\sup\{J(t_0,x_0,\hat{\mathfrak{u}}_\Delta,v): d(\Delta)\leq\delta,v\in\mathcal{V}\}\leq \psi_\sigma(t_0,x_0)+RC_1\sigma. $$
\end{corollary}

\begin{corollary}\label{cor_rate}
  $|{\rm Val}(t_0,x_0)-\psi_\sigma(t_0,x_0)|\leq RC_1\sigma.$
\end{corollary}

\begin{Remark} Corollary \ref{cor_rate} provides the rate of convergence for the vanishing viscosity approximations of Hamilton--Jacobi PDE. This result corresponds to \cite[Proposition~3.2]{Camilli}.
\end{Remark}

\section{Approximation of differential game by Markov games}
Given a differential game with  dynamics (\ref{system}) define the Markov game in the following way.

Let $h$ be a positive number, $f^1(t,x,u,v)=(f^1_1(t,x,u,v),\ldots,f_d^1(t,x,u,v))$ and let $e^i$ denote the $i$-th coordinate vector. Put $$\chi_i(t,x,u,v)=\left\{
\begin{array}{cc}
  e^i, & f_i^1(t,x,u,v)>0, \\
  -e^i, & f_i^1(t,x,u,v)<0, \\
  0, & f_i^1(t,x,u,v)=0.
\end{array}
\right.$$
For $A\subset\mathbb{R}^d$
$$\nu^2(t,x,u,v,A)\triangleq\frac{1}{h}\sum_{i=1}^n|f_i(t,x,u,v)|\delta_{h\chi_i(t,x,u,v)}(A). $$
Recall that $\delta_z$ denotes the Dirac measure concentrated at $z$.

Further, define
\begin{multline}\label{l2_def_markov}L^2_t[u,v]\varphi(x)\triangleq\intrd [\varphi(x+y)-\varphi(x)]\nu^2(t,x,u,v,dy)\\ =\sum_{i=1}^n|f_i(t,x,u,v)|\frac{\varphi(x+h\chi_i(t,x,u,v))-\varphi(x)}{h}. \end{multline}

This generator corresponds to the continuous-time Markov chain on $h\mathbb{Z}^d$ with the Kolmogorov matrix
\begin{equation}\label{kolmogorov_mtr_approx}
Q_{xy}^h(t,u,v)=\left\{
\begin{array}{lr}
  \frac{1}{h}|f_i(t,x,u,v)|, & y=x+h\chi_i(t,x,u,v), \\
  -\frac{1}{h}\sum_{i=1}^d|f_i(t,x,u,v)|, & x=y, \\
  0, & y\neq x,\ \  y\neq x+h\chi_i(t,x,u,v),
\end{array}
\right.
\end{equation}
The value function for the game with the generator $L^2$ defined by (\ref{l2_def_markov}) provides the upper and lower bounds for the value function of the differential game.

The following system of ODEs is the Isaacs--Bellman equation for the Markov game.
\begin{multline}\label{zachrisson_upper}
  \frac{d}{dt}\eta_h^+(t,x)+\min_{u\in U}\max_{v\in V}\sum_{i=1}^d|f_i(t,x,u,v)|\frac{\eta^+_h(t,x+h\chi_i(t,x,u,v))-\eta^+_h(t,x)}{h}=0, \\ \eta^+_h(T,x)=g(x)
\end{multline}  where $x\in h\mathbb{Z}^d$ is a parameter.
\begin{Pl}\label{pl_existence}There exists an unique solution of (\ref{zachrisson_upper}).
\end{Pl}
\begin{proof} We consider  system (\ref{zachrisson_upper}) as a differential equation in the Banach space
$$\mathcal{M}=\left\{\varrho:h\mathbb{Z}^d\rightarrow \mathbb{R}\mbox{ with } \|\varrho\|_{\mathcal{M}}<\infty\right\} $$ where
$$\|\varrho\|_{\mathcal{M}}=\sup_{x\in h\mathbb{Z}^d}\frac{|\varrho(x)|}{h+\|x\|}. $$

First let us show that if $\eta_h^+$ solves  (\ref{zachrisson_upper}) then $\eta_h^+(t,\cdot)$ belongs to $\mathcal{M}$ for any $t\in [0,T]$. We have that
$$\frac{|\eta_h^+(t,x)|}{h+\|x\|}\leq \frac{|g(x)|}{h+\|x\|}+dM_1^1\int^{T}_t\left[\frac{|\eta_h^+(\tau,x)|}{h+\|x\|}+ \frac{|\eta_h^+(\tau,x+h\chi_i(t,x,u,v))|}{h+\|x\|}\right]d\tau.  $$
Since $2(h+\|x\|)\geq 2h+\|x\|\geq h+\|x+h\chi_i(t,x,u,v))\|$, we get the inequality
$$\|\eta_h^+(t,\cdot)\|_{\mathcal{M}}\leq \|g(\cdot)\|_{\mathcal{M}}+3dM_1^1\int^{T}_t\|\eta_h^+(\tau,\cdot)\|_{\mathcal{M}}d\tau.  $$ From Lipschitz continuity of the function $g$ and Gronwall's inequality it follows that
for $t\in [0,T]$
$$\|\eta_h^+(t,\cdot)\|_{\mathcal{M}}<\infty. $$ Thus, if $\eta^+_h$ solves (\ref{zachrisson_upper}) then $\eta_h^+(t,\cdot)\in\mathcal{M}$.

Define the mapping $\mathcal{H}:[0,T]\times \mathcal{M}\rightarrow \mathcal{M}$ by the rule
$$\mathcal{H}[t,\varrho](x)=\min_{u\in U}\max_{v\in V}\sum_{i=1}^d|f_i(t,x,u,v)|\frac{\varrho(x+h\chi_i(t,x,u,v))-\varrho(x)}{h}. $$
Consider the boundary value problem
\begin{equation}\label{zach_revised}
\frac{d}{dt}\varrho[t]=-\mathcal{H}[t,\varrho[t]], \ \ \varrho[T](x)=g(x).
\end{equation}
We have that the function $\mathcal{H}$ is continuous and Lipschitz continuous w.r.t. $\varrho$. Indeed, for $t,s\in [0,T]$, $\varrho,\varrho',\varrho''\in\mathcal{M}$
$$ \|\mathcal{H}[t,\varrho]-\mathcal{H}[s,\varrho]\|_{\mathcal{M}}\leq 3\alpha_1^1(t-s)\|\varrho\|_\mathcal{M},\ \ \|\mathcal{H}[t,\varrho']-\mathcal{H}[t,\varrho'']\|\leq 3 M_1^1\|\varrho'-\varrho''\|_{\mathcal{M}}. $$ Hence, by \cite[\S 1]{deimling}  problem (\ref{zach_revised})
has an unique solution $\varrho^*[\cdot]$.

 Put $\eta_h^+(t,x)\triangleq\varrho^*[t](x)$.
The function $\eta_h^+$ is an unique solution of (\ref{zachrisson_upper}).
\end{proof}
\begin{Th}\label{th_upper_bound}
There exists a constant $C_2$ determined by the function $f^1$ such that if $\eta_h^+$ is a solution of (\ref{zachrisson_upper}) then for $t_0\in[0,T]$, $x_0\in h\mathbb{Z}^d$
    \begin{equation}\label{ineq_approx_divf}
    |{\rm Val}(t_0,x_0)-\eta_h^+(t_0,x_0)|\leq RC_2\sqrt{h}.
    \end{equation}
\end{Th}
\begin{proof}
First, let us show that $\eta^+_h$ is an upper value  of the Markov game with the Kolmogorov matrix defined by (\ref{kolmogorov_mtr_approx}). If $\mathfrak{u}$ and $\mathfrak{v}$ are strategies of the first and second players respectively then denote the upper outcome in the game with the generator $L^2$ given by (\ref{l2_def_markov}) by $\mathcal{I}^*(t_0,x_0,\mathfrak{u},\mathfrak{v})$.

Let $\mathfrak{u}=(\Omega^U,\mathcal{F}^U,\{\mathcal{F}^U_s\}_{s\in [t_0,T]},u_{x(\cdot)},P^U_{x(\cdot)})$ be a strategy of the first player. Consider the following counter-strategy of the second player $\bar{\mathfrak{v}}[\mathfrak{u}]=(\Omega^V,\mathcal{F}^V,\{\mathcal{F}^V_s\}_{s\in [t_0,T]},v_{x(\cdot)},P^V_{x(\cdot)})$, where
$$\Omega^V\triangleq\Omega^U,\ \ \mathcal{F}^V\triangleq\mathcal{F}^U\ \ \mathcal{F}^V_s\triangleq\mathcal{F}^U_s\ \ P^V_{x(\cdot)}\triangleq P_{x(\cdot)}^U, $$
$$v_{x(\cdot)}(t,\omega^U)\triangleq \bar{v}_x(t,u_{x(\cdot)}(t,\omega^U)),$$
$$\bar{v}_x(t,u)\in {\rm Argmax}\left\{\sum_{i=1}^d f_i^1(t,x(t),u,v)\eta^+_h(t,x(t)):v\in V\right\}. $$ Note that the Markov chain generated by the pair of strategies $\mathfrak{u}$, $\bar{\mathfrak{v}}[u]$ corresponds  to the controlled Markov chain with the Kolmogorov matrix
$Q_{xy}^h(t,u,\bar{v}_{x}(t,u))$.

Using dynamical programming arguments \cite[Theorem 8.1]{fleming_soner}  we obtain that
\begin{multline}\label{lower_eta} \min_{\mathfrak{u}}\max_{\mathfrak{v}}\mathcal{I}^*(t_0,x_0,\mathfrak{u},\mathfrak{v})
= \min_{\mathfrak{u}}\mathcal{I}^*(t_0,x_0,\mathfrak{u},\bar{\mathfrak{v}}[\mathfrak{u}])\\\leq
\mathcal{I}^*(t_0,x_0,\mathfrak{u}^*,\bar{\mathfrak{v}}[\mathfrak{u}^*])=\eta^+_h(t_0,x_0),\end{multline} where
$\mathfrak{u}^*=(\Omega^{*,U},\mathcal{F}^{*,U},\{\mathcal{F}^{*,U}_s\}_{s\in [t_0,T]},u_{x(\cdot)}^*,P^{*,U}_{x(\cdot)})$ is
such that $$u^*_{x(\cdot)}(t)=u^*(t,x(t))\in {\rm Argmin}\left\{\max_{v\in V}\sum_{i=1}^df^i(t,x,u,v)\eta_h^+(t,x):u\in U\right\}.$$
Further, consider the controlled Markov chain with the Kolmogorov matrix
$Q_{xy}^h(t,u^*(t,x),v)$. This system corresponds to the case when the first player uses the strategy $\mathfrak{u}^*$. Using dynamical programming arguments once more time  we obtain that
\begin{multline}\label{upper_eta} \min_{\mathfrak{u}}\max_{\mathfrak{v}}\mathcal{I}^*(t_0,x_0,\mathfrak{u},\mathfrak{v})
= \max_{\mathfrak{v}}\mathcal{I}^*(t_0,x_0,\mathfrak{u}^*,\mathfrak{v})\\\geq
\mathcal{I}^*(t_0,x_0,\mathfrak{u}^*,\bar{\mathfrak{v}}[\mathfrak{u}^*])=\eta^+_h(t_0,x_0).\end{multline}
Combining (\ref{lower_eta}) and (\ref{upper_eta}) we obtain that $\eta^+_h$ is an upper value function for the game with the generator defined by (\ref{l2_def_markov}).

Now let us show that $\eta_h^+$ is $u$-stable with respect to $L^2$ given by (\ref{l2_def_markov}).
Let $t,\theta\in [0,T]$, $t<\theta$. We have that if the initial position of the Markov chain with  Kolmogorov matrix (\ref{kolmogorov_mtr_approx}) belongs to $h\mathbb{Z}^d$ then the state of this Markov chain belongs to $h\mathbb{Z}^d$. Put
$$\widetilde{\Omega}^{t,\theta}\triangleq D([t,\theta],h\mathbb{Z}^d),\ \ \widetilde{\mathcal{F}}_s^{t,\theta}\triangleq \mathcal{B}(D([t,s],h\mathbb{Z}^d)), \ \ \widetilde{\mathcal{F}}^{t,\theta}\triangleq\widetilde{\mathcal{F}}_s^{t,\theta}. $$

For a given $v\in V$, $\xi\in h\mathbb{Z}^d$, $\omega\in\widetilde{\Omega}^{t,\theta}$ put $$u_{\xi,v}^{t,\theta}(\tau,\omega)\triangleq u^*(\tau,\omega(\tau)). $$ Further, let $\widetilde{P}_{\xi,v}^{t,\theta}$ be a probability of the Markov chain with the Kolmogorov matrix
$Q^h_{xy}(\tau,u^*(\tau,x),v)$. Let $Y_{\xi,v}^{t,\theta}(\tau)$ be a state at time $\tau$ of the Markov chain starting at $(t,\xi)$. This means that for any $\varphi\in\mathcal{D}^2$
$$\varphi(Y_{\xi,v}^{t,\theta}(s))- \int_{t}^sL^2_\tau[u_{\xi,v}^{t,\theta}(\tau),v]\varphi(Y^{t,\theta}_{\xi,v}(\tau))d\tau $$ is a martingale. The dynamic programming principle  yields that
$$\eta_h^+(t,\xi)\geq \widetilde{\mathbb{E}}^{t,\theta}_{\xi,v}\eta^+(\theta,Y^{t,\theta}_{\xi,v}(\theta)). $$
Setting $$\hat{\mu}^{t,\theta}_{\xi,v}\triangleq \delta_{u^{t,\theta}_{\xi,v}} $$
we obtain that the function $\eta_h^+$ is $u$-stable with respect to the generator $L^2$ given by~(\ref{l2_def_markov}).

Using (\ref{kappa_def}) and (L5) we conclude that for the generators $L^1$ and $L^2$ given by (\ref{L1_diff}) and (\ref{l2_def_markov}) respectively  $$\varkappa=0, \ \  M_0^1=0,$$
\begin{multline*}M_0^2=\sup_{t\in [0,T],x\in\mathbb{R}^d,u\in U,v\in V}\intrd y^2\nu^2(t,x,u,v,dy)\\=\sup_{t\in [0,T],x\in\mathbb{R}^d,u\in U,v\in V}\sum_{i=1}^d h|f_i(t,x,u,v)|\leq d^{3/2}M_1^1h.\end{multline*} Hence, Corollary~\ref{crl_gamma_1} yields the inequality
\begin{equation}\label{ineq_approx_diff_upper}
{\rm Val}(t_0,x_0)\leq\eta_h^+(t_0,x_0)+RC_2\sqrt{h}
 \end{equation}
for any $(t_0,x_0)\in[0,T]\times h\mathbb{Z}$, and the constant  $C_2\triangleq d^{3/4}(M_1^1)^{1/2}C$, where $C$ is defined by (\ref{C_def}).

Now we shall prove the following inequality
\begin{equation}\label{ineq_approx_diff_lowe}
\eta_h^+(t_0,x_0)\leq {\rm Val}(t_0,x_0)+RC_2\sqrt{h}.
 \end{equation} Consider the generators $\widehat{L}^1_t[u,v]=L^2_t[u,v]$ and $\widehat{L}^2_t[u,v]=L^1_t[u,v]$ where $L^1$ and $L^2$ are determined by (\ref{L1_diff}) and (\ref{l2_def_markov}) respectively. We have that the pair of generators $\widehat{L}^1$ and $\widehat{L}^2$ satisfies conditions (L1)--(L8). Note that in this case condition (L8)--(2) is fulfilled (see (\ref{isaacs_diff})). The function ${\rm Val}$ is $u$-stable with respect to the generator $\widehat{L}^2$ (see Proposition \ref{pl_u_stable_1})). Since $\eta_h^+$ is an upper value function for the Markov game, using Corollary \ref{crl_gamma_1} we get inequality (\ref{ineq_approx_diff_lowe}).

 Combining (\ref{ineq_approx_diff_upper}) and (\ref{ineq_approx_diff_lowe}) we get inequality (\ref{ineq_approx_divf}).
\end{proof}

\begin{corollary}\label{crl_lower_bound}
 For $t_0\in [0,T]$, $x_0\in h\mathbb{Z}^d$
    $$
    |{\rm Val}(t_0,x_0)-\eta_h^-(t_0,x_0)|\leq RC_2\sqrt{h}.
   $$
\end{corollary}
\begin{proof}
To prove this Corollary it suffices to interchange the players and replace the payoff function with  $-\sigma$ in Theorem \ref{th_upper_bound}.
\end{proof}


\begin{thebibliography}{99}
\bibitem{averboukh_dgaa} Averboukh Yu. Extremal shift rule for continuous-time zero-sum Markov games. Dyn Games Appl (in press). Preprint at Arxiv:1412.0643 (2016)

\bibitem{bardi} Bardi M, Capuzzo Dolcetta I.  Optimal control and viscosity solutions of Hamilton--Jacobi--
Bellman equations. {Birkhäuser}, Basel (1996)

\bibitem{Billingsley}  Billingsley P. Convergence of probability measures. Wiley,  New  York (1999)


\bibitem{buckdahn_li} Buckdahn R, Li J. Stochastic differential games and viscosity solutions of Hamilton--Jacobi--Bellman--Isaacs equations. SIAM J Control Optim 47:1 (2008) 444--475

\bibitem{buslaeva} Buslaeva L. Stochastic control in a differential game. J Appl Math Mech  42:4 (1977) 609-623

\bibitem{Camilli}  Camilli F,  Marchi C. Continuous dependence estimates and homogenization of quasi-monotone systems of fully nonlinear second order parabolic equations. Nonlinear Anal-theor  75:13 (2012) 5103--5118

\bibitem{deimling}  Deimling K. Ordinary differential equations in Banach spaces. Lecture Notes in Mathematics, 596. Springer, Berlin (1977)


\bibitem{elliot_stochastic} Elliott R. The existence of value in stochastic differential games. SIAM J Control Optim 14:1 (1976) 85--94.

\bibitem{elliot_kalton} Elliott RJ,   Kalton NJ. Values in differential games. Bull Amer Math Soc 78:3  (1972) 427-431

\bibitem{evans_souganidis} Evans LC, Souganidis PE. Differential games and representation formulas for solutions of
Hamilton--Jacobi equations. Indiana U Math J 282 (1984) 487--502.

\bibitem{fleming_soner} Fleming WH,  Soner HM.  Controlled Markov processes and viscosity solutions. {Springer, New York} (2006)


\bibitem{friedman}  Friedman A. Existence of value and of saddle points for differential games of pursuit and evasion.
 J Differ. Equations
 7:1 (1970) 92--110

 \bibitem{friedman_parabolic} Friedman A.   Partial differential equations of parabolic type. Prentice-Hall, Englewood Cliffs (1964)


\bibitem{GihmanSkorokhod} Gihman II,  Skorohod AV. Controlled stochastic processes. Springer, New  York (1979)


\bibitem{Guo_Hernandez_lerma} Guo X, Hern\'{a}ndez-Lerma O.  Zero-sum games for
continuous-time Markov chains with unbounded transitions and average payoff rates. J Appl Probab. 40 (2003) 327--345.



\bibitem{Hamadene} Hamad\`{e}ène S. Backward-forward SDE's and stochastic differential games.
Stoch Proc Appl
 77:1 (1998)  1--15

 \bibitem{Hamadene_1995} Hamad\`{e}ène S, Lepeltier JP. Backward equations, stochastic
control and zero-sum stochastic
differential games. Stoch Stoch Rep
54:3-4 (1995) 221--231

\bibitem{Kol_markov}  Kolokoltsov VN. Markov processes, semigroups and generators.
De Gruyter Studies in Mathematics 38, De Gryuter (2011)

\bibitem{Kol} Kolokoltsov VN. Nonlinear Markov games on a finite state space (Mean-field and
binary interactions). Int J Stat Probab  1:1 (2012) 77--91




\bibitem{a4}	Krasovskii NN, Kotelnikova AN.	 An approach-evasion differential game: stochastic guide. P Steklov Inst Math  269:1 Supplement (2010) 191--213



\bibitem{a5} Krasovskii NN, Kotelnikova AN. On a differential interception game. P Steklov Inst Math 268:1 (2010)  161--206

\bibitem{kras_delay} Krasovskii NN, Kotelnikova AN. Stochastic guide for a time-delay object in a positional differential game. P Steklov Inst Math  277:1 Supplement (2012) 145-151

\bibitem{a6} Krasovskii NN, Kotelnikova AN. Unification of differential games, generalized solutions of the Hamilton-Jacobi equations, and a stochastic guide. Diff Equat+  45:11 (2009)  1653--1668	

\bibitem{KrasSubb_guide} Krasovskii NN,  Subbotin AI. Approximation in a differential game. J Appl Math Mech 37:2 (1973)  185--192

\bibitem{NN_PDG_en} Krasovskii NN,  Subbotin AI.  Game-theoretical control
problems. Springer, New  York (1988)

\bibitem{luk_plaks}  Lukoyanov NYu,  Plaksin AR. Finite-dimensional modeling guides in time-delay systems. Trudy Instituta Matematiki i Mekhaniki UrO RAN 19, No. 1 (2013)  182--195 (in Russian)

\bibitem{SDG_book} Ramachandran KM, Tsokos CP. Stochastic differential games. Theory and applications. {Atlantis Press}. Paris--Amsterdam--Beijing (2012)

\bibitem{Subb_book} Subbotin AI.  Generalized solutions of first-order PDEs. The dynamical perspective. 
{Birkha\"{u}ser, Boston} (1995)

\bibitem{subb_chen} Subbotin AI, Chentsov AG. Optimization of guarantee in control problems. Nauka, Moscow (1981, in Russian).

\bibitem{varaya_lin}  Varaya P,  Lin J.
Existence of saddle points in differential games
Siam J Control 7:1 (1969) 142--157

\bibitem{Warga}  Warga J. Optimal control of differential and functional equations. Academic press, New York (1972)


\bibitem{Zachrisson} Zachrisson LE. Markov games. In: Dresher M, Shapley LS, Tucker AW (eds) Advances in game theory. Princeton University Press, Princeton (1964). 211--253































\end{thebibliography}
\end{document}